\documentclass[12pt]{amsart}

\usepackage{amssymb}
\usepackage{enumerate}
\usepackage{graphicx}

\makeatletter
\@namedef{subjclassname@2010}{
  \textup{2010} Mathematics Subject Classification}
\makeatother

\newtheorem{thm}{Theorem}[section]

\newtheorem{lem}[thm]{Lemma}
\newtheorem{prop}[thm]{Proposition}

\theoremstyle{definition}
\newtheorem{defin}[thm]{Definition}
\newtheorem{rem}[thm]{Remark}

\newcommand{\N}{\mathbb{N}}

\newcommand{\R}{\mathbb{R}}
\newcommand{\la}{\lambda}

\newcommand{\rk}{\operatorname{rk}}

\newcommand{\pr}{\operatorname{pr}}

\newcommand{\ind}{\operatorname{ind}}

\newcommand{\diag}{\mathrm{diag}}
\newcommand{\spn}{\mathrm{span}}

\newcommand{\II}{\mathfrak{I}}
\newcommand{\GG}{\mathcal{G}}
\newcommand{\g}{\mathfrak{g}}

\begin{document}

\title[sub-pseudo-Riemannian isometries]{On contact sub-pseudo-Riemannian isometries}

\author[M. Grochowski]{Marek Grochowski}
\address{Faculty of Mathematics and Natural Sciences\\ Cardinal Wyszy\'nski University\\
00-956 Warszawa, Poland}
\email{m.grochowski@uksw.edu.pl}

\author[W. Kry\'nski]{Wojciech Kry\'nski}
\address{Institute of Mathematics, Polish Academy of Sciences, ul.~\'Sniadeckich 8, 00-956 Warszawa, Poland}
\email{w.krynski@impan.pl}

\date{\today}

\begin{abstract}
We study isometries in the contact sub-pseudo-Rie\-man\-nian geometry. In particular we give an upper bound on the dimension of the isometry group of a general sub-pseudo-Riemannian manifold and prove that the maximal dimension is attained for the left invariant structures on the Heisenberg group.
\end{abstract}

\maketitle

\section{Introduction}

\subsection{Results}
Let $M$ be a smooth connected manifold. A \emph{sub-pseudo-Riemannian structure} on $M$ is a couple $(H,g)$ made up of a smooth bracket generating 
distribution $H$ of constant rank and a smooth pseudo-Riemannian metric $g$ on $H$. At each point $q\in M$, $g$ can be represented as a diagonal matrix
\begin{align*}
\diag(-1,\dots,-1,+1,\dots,+1)
\end{align*}
with, say, $l$ minuses. Clearly, by continuity, the number $l$ does not depend on a point $q$. It will be denoted $\ind(g)$ and called \emph{the index of 
the metric $(H,g)$}. 

A triple $(M,H,g)$ is called a \emph{sub-pseudo-Riemannian manifold}. In particular, if $\ind(g)=0$ then $(M,H,g)$ is called a 
\emph{sub-Riemannian  manifold}. This case is best known and there are a lot of papers and books devoted to the sub-Riemannian geometry 
(see \cite{A,ACG,AB,Gr,M} and references therein). If $\ind(g)=1$ then $(M,H,g)$ is called a \emph{sub-Lorentzian manifold} (see \cite{G1,GW,GV,KM}). 
The sub-pseudo-Riemannian structures can be interpret as control systems \cite{A,G0}. In particular the sub-Lorentzian structures give rise to a class of 
control-affine systems (c.f. \cite{G0,G1}).

In the sub-pseudo-Riemannian geometry we can ask the same questions as in the classical pseudo-Riemannian geometry. One of the most fundamental problems 
considered in the pseudo-Riemannian geometry is connected to calculations of the isometry group of a given pseudo-Riemannian manifold. We shall consider 
a generalisation of this problem to the sub-pseudo-Riemannian case.

\begin{defin}
Fix a sub-pseudo-Riemannain manifold $(M,H,g)$. A diffeomorphism $f\colon M\to M$ is called an isometry if 
\begin{enumerate}
\item[(D1)] $f$ preserves the distribution, i.e. $f_*(H)=H$;
\item[(D2)] $f_*\colon H_q \to H_{f(q)}$ is a linear isometry for every $q\in M$, i.e. $g(f_*(v),f_*(w))=g(v,w)$ for all $v,w\in H_q$.
\end{enumerate}
\end{defin}
The set of all isometries is a group (in fact a Lie group as it will become clear soon) and will be denoted $\II(M,H,g)$. 
The component of the identity of this group
is $\II_0(M,H,g)$. Clearly $\dim\II(M,H,g)=\dim\II_0(M,H,g)$. We shall assume that $H$ is a contact distribution meaning that it is locally given by the kernel of a contact 
one-form $\alpha$ satisfying
\begin{equation}\label{eq1}
(d\alpha)^{\wedge n}\wedge\alpha\neq 0,
\end{equation}
where $\dim M=2n+1$. In this case $(M, H, g)$ will be referred to as a \emph{contact sub-pseudo-Riemannian manifold}. Our main result is the following
\begin{thm}\label{thm1}
Let $(M,H,g)$ be a contact sub-pseudo-Riemannian manifold. If $\ind(g)$ is even or $\ind(g)=\frac{1}{2}\rk H$ then
\begin{equation}\label{est1}
\dim \II(M,H,g)\leq \dim M+\left(\frac{1}{2}\rk H\right)^2.
\end{equation}
If $\ind(g)$ is odd and $\ind(g)\neq\frac{1}{2}\rk H$ then 
\begin{equation}\label{est2}
\dim \II(M,H,g)\leq \dim M+\left(\frac{1}{2}\rk H-1\right)^2+1.
\end{equation}
\end{thm}

In Section \ref{sec3}, Proposition \ref{prop4}, we will show that the maximal dimension in \eqref{est1} and \eqref{est2} is attained by a left-invariant structure on the Heisenberg group. More precisely we will show that for any value of $\ind(g)\in\{0,1,\ldots,\rk H\}$ and any $t\leq\min\{\ind(g),\rk H-\ind(g)\}$ such that $\ind(g)-t$ is even there is a left-invariant structure such that 
\begin{equation}\label{est3}
\dim \II(M,H,g)=\dim M+\left(\frac{1}{2}\rk H-t\right)^2+t^2.
\end{equation}
In particular the maximal dimension in Theorem \ref{thm1} is attained for $t=0$ if $\ind(g)$ is even, for $t=\ind(g)$ if $\ind(g)=\frac{1}{2}\rk H$ and for $t=1$ if $\ind(g)$ is odd and not equal $\frac{1}{2}\rk H$.

Let us point out here that invariants for the contact sub-pseudo-Riemannian structures has been recently constructed in \cite{GK} (see also \cite{A} for the sub-Riemannian case). The invariants vanish for the left-invariant structures satisfying \eqref{est3}. 

\subsection{Connections with control systems.}\label{sec1.2}
Suppose that $(\Sigma)$ $\dot{q}=f(q,u)$ is a control system on a manifold $M$. By a symmetry of $(\Sigma)$ we mean a diffeomorphism of $M$ which maps 
the trajectories of $(\Sigma)$ onto trajectories of $(\Sigma)$. It turns out that the described results concerning isometry groups of sub-pseudo-Riemannian 
manifolds can be formulated in terms of symmetries of certain control systems.

Indeed, suppose that $(M,H,g)$ is a sub-pseudo-Riemannian manifold with $\ind(g)=l$ and $\rk H=k$. By a timelike (resp. spacelike) curve on $(M,H,g)$ 
we mean every absolutely continuous curve $\gamma\colon [a,b] \to M$ such that $\dot{\gamma}\in H_{\gamma(t)}$ and moreover 
$g(\dot{\gamma}(t),\dot{\gamma}(t))<0$ (resp. $g(\dot{\gamma}(t),\dot{\gamma}(t))> 0$) for almost every $t\in [a,b]$. Suppose that $X_1,\dots,X_k$ is 
an orthonormal basis for $(H,g)$ defined on an open set $U$ such that $g(X_i,X_i)=-1$ for $i=1,\dots,l$ and $g(X_i,X_i)=1$ for $i=l+1,\dots,k$. 
Timelike (resp. spacelike) curves in $U$ with unit speed parametrization can be represented as solutions to the following control system
\begin{equation}\label{ConSys}
\dot{q}=\sum_{i=1}^{k} u_i X_i(q),
\end{equation}
with the set of control parameters equal to
$$
\mathcal{U}=\left\{(u_1,\dots,u_k)\ |\text{\;} -\sum_{i=1}^{l}u_i^2+\sum_{i=l+1}^{k}u_i^2 = -1 \right\},
$$
or
$$
\mathcal{U}=\left\{(u_1,\dots,u_k)\ |\text{\;} -\sum_{i=1}^{l}u_i^2+\sum_{i=l+1}^{k}u_i^2 = 1 \right\},
$$
respectively, where controls are supposed to be measurable and bounded.  Now, it is easy to show that in both cases the symmetries of \eqref{ConSys} 
coincide with the isometry group $\II(U,H_{|U},g)$. One can also consider the sets of null, or nonspecelike curves defined by similar control systems. 
However in the latter cases the isometry group  $\II(U,H_{|U},g)$ is only a subgroup of all symmetries.

In the sub-Lorentzian setting the future directed nonspacelike curves can be described by a control-affine system. To be precise, by a time orientation of a 
sub-Lorentzian manifold $(M,H,g)$ we understand a timelike vector field $X$ on $M$ (i.e. $X(q)\in H_q$ and $g(X(q),X(q))<0$ for every $q\in M$). 
A nonspacelike curve $\gamma\colon [a,b] \to M$ is said to be future directed if $g(\dot{\gamma}(t),X(\gamma(t)))<0$ a.e. on $[a,b]$ (c.f. \cite{G0, G1}). 
Suppose that $X$ is a fixed time orientation and $U$ is an open set on which there exist 
spacelike vector fields $X_2,\dots,X_k$ such that $X,X_2,\dots,X_k$ form an orthonormal basis for $(H,g)$ over $U$. As it is explained in \cite{G0} every 
nonspacelike future directed curve in $U$ is, up to a reparameterization, a trajectory of the control-affine system
\begin{align}\label{AffConSys}
\dot{q}=X+\sum_{i=2}^{k}u_iX_i(q)\text{,}
\end{align}
where the set of control parameters equals the unit ball in $\mathbb{R}^{k-1}$ centered at zero. Now it is clear that $\II(U,H_{|U},g)$ is a group of 
symmetries of the system (\ref{AffConSys}). We refer to \cite{A,G0,G1} for more information on 
the mentioned control systems, the corresponding reachable sets and optimal solutions to the control problems.

\subsection{The content of the paper.}
The paper is organised as follows. In Section \ref{sec2} we formulate and explain basic facts and assumptions that we use later on. We show that $g$ can 
be extended to a metric on $TM$ in a canonical way and exploit this fact to prove that $\II(M,H,g)$ is a Lie group (Theorem \ref{thm2}). Moreover, we 
introduce a canonical symplectic structure on $H$.

Sections \ref{sec3} and \ref{sec4} are devoted to special classes of sub-pseudo-Rie\-man\-nian metrics. In Section \ref{sec3} we consider 
sub-pseudo-Riemannian structures satisfying an additional compatibility condition. In the Riemannian signature the condition guarantees that $H$ caries an 
almost complex structure. In Section \ref{sec4} we consider so-called regular structures, which include all sub-Riemannian and sub-Lorentzian metrics in 
neighbourhoods of generic points. We estimate from above dimensions of the isometry groups for these spacial classes of structures (Theorems \ref{thm3} 
and \ref{thm4}). Moreover, we construct examples with isometry groups of dimension given by formula \eqref{est3}.

Section \ref{sec5} contains the proof of Theorem \ref{thm1}. The main idea relies on the calculation of the Tanaka prolongations of certain graded Lie 
algebras and on the results of Kruglikov \cite{Kr,Kr2} that extend Tanaka's theory to the case of non-constant symbol algebras.


\section{Contact sub-pseudo-Riemannian structures}\label{sec2}

\subsection{Extended metric}
Let $(M,H,g)$ be a contact sub-pseudo-Rie\-man\-nian manifold of dimension $2n+1$. Fix $q\in M$ and assume that $H=\ker\alpha$ in a neighbourhood of $q$, where $\alpha$ satisfies \eqref{eq1}. The contact form $\alpha$ defines the Reeb vector field $X_\alpha$ by the conditions
\begin{equation}\label{reeb}
X_\alpha\in\ker d\alpha,\qquad \alpha(X_\alpha)=1.
\end{equation}
It follows that $X_\alpha$ is transverse to $H$. Clearly $X_\alpha$ depends essentially on the choice of $\alpha$ and the one-form is not unique. However it can be normalised in the following way. Let $(X_1,\ldots,X_{2n})$ be an orthonormal frame of $H$ in a neighbourhood of $q$. Then, multiplying $\alpha$ by a smooth function, we can impose the condition
\begin{equation}\label{norm1}
|(d\alpha)^{\wedge n}(X_1,\ldots,X_{2n})|=1,
\end{equation}
which does not depend on the choice of an orthonormal frame. As a result, we get a canonical contact form $\alpha$ given up to a multiplication by $\pm 1$ in the neighbourhood of $q\in M$. We shall see later that for oriented structures one can rid off this ambiguity and get a unique canonical global contact form $\alpha$ on $M$. However, we do not need the uniqueness at this point and using the two normalised contact forms in a neighbourhood of any point $q\in M$ we are able to extend $g$ from $H$ to a metric $G$ on $TM$. Indeed, we set
$$
G|_{H\times H}=g
$$
and
$$
G(X_\alpha,X_\alpha)=1,\qquad G(X_\alpha, H)=0,
$$
where $\alpha$ is a contact form satisfying \eqref{norm1} and $X_\alpha$ is the Reeb vector field corresponding to $\alpha$. Since $\alpha$ is given up 
to a sign, we conclude that $X_\alpha$ is given up to a sign too. However, $G$ does not depend on the sign and we obtain unique $G$ in a neighbourhood of 
each point $q\in M$. The uniqueness implies that $G$ must coincide on overlaps of neighbourhoods of different points. Thus, we get a globally defined 
metric $G$ on $M$ which is canonically determined by the structure $(H,g)$. 
Since any isometry preserves the form $\alpha$ up to a sign, we have proved the following
\begin{prop}\label{prop1}
If $f\colon M\to M$ is an isometry of a sub-pseudo-Riemannian structure $(H,g)$ then $f^*G=G$. Thus $f$ is an isometry of $G$, too.
\end{prop}
We shall denote by $\II(M,G)$ the group of isometries of $(M,G)$.
We refer to \cite{GK} for more detailed discussion on the possible extensions of $g$.

Let $O_G(M)$ be the orthonormal frame bundle for $G$. We define $O_{H,g}(M)$, the \emph{orthonormal frame bundle} of $(H,g)$, as a sub-bundle of $O_G(M)$ consisting of points $(q; v_1,\ldots, v_{2n}, v_0)$ such that $(v_1,\ldots,v_{2n})$ is an orthonormal basis of $H_q$. In particular, it follows that $v_0=X_\alpha(q)$ where $\alpha$ is one of the two contact forms normalised by \eqref{norm1} in a neighbourhood of $q$. Now, any pseudo-Riemannian isometry $f\in\II(M,G)$ is uniquely determined by the values of $f(q)$ and $f_*(q)$ where $q$ is an arbitrary fixed point in $M$ \cite{Ko}. 
Since $\II(M,H,g)$ is a closed subgroup of $\II(M,G)$ we get  

\begin{thm}\label{thm2}
$\II(M,H,g)$ is a Lie group with respect to the open-compact topology. Moreover any contact sub-pseudo-Riemannian isometry $f\in \II(M,H,g)$ is uniquely 
determined by two values: $f(q)$ and $f_*(q)$, where $q\in M$ is an arbitrarily fixed point. Additionally, fixing an arbitrary point 
$(q; v_1,\ldots, v_{2n}, v_0)\in O_{H,g}(M)$, the mapping
\begin{equation}\label{embed1}
f\longmapsto(f(q); f_*(v_1),\ldots,f_*(v_{2n}),f_*(v_0))
\end{equation}
defines an embedding of $\II(M,H,g)$ to $O_{H,g}(M)$.
\end{thm}
\begin{proof}
Follows from the fact that $\II(M,G)$ is a Lie group \cite{Ko} and its subgroup $\II(M,H,g)$ is closed in $\II(M,G)$.
\end{proof}

\subsection{Orientation}
Let $(M,H,g)$ be a contact sub-pseudo-Rie\-man\-nian manifold of dimension $2n+1$. We shall say that the structure is \emph{oriented} if the two vector 
bundles $TM$ and $H$ are oriented (see \cite{G1} for various notions of orientations related to the casual decomposition of the distribution under 
consideration). 
We shall see that the structure is oriented if and only if there is a global contact form annihilating $H$. There are two cases depending on the parity of 
$n$.

If $n$ is even then $(d\alpha)^{\wedge n}$ is independent of the sign of $\alpha$. Conversely, the sign of $d\alpha^{\wedge n}\wedge\alpha$ changes if 
the sign of $\alpha$ changes. Thus, on the one hand, $H$ is canonically oriented, because fixing an open cover $\{U_s\}_{s\in\Sigma}$ of $M$ and local 
contact forms $\{\alpha_s\}_{s\in\Sigma}$ annihilating $H$ on $U_s$ we can rescale the forms such that $(d\alpha_s)^{\wedge n}$ glue to a global $2n$-form 
non-degenerate on $H$. On the other hand, $M$ is oriented if and only if there is a global contact form annihilating $H$. Indeed, if $\alpha$ is a global 
contact form then $d\alpha^{\wedge n}\wedge\alpha$ defines an orientation of $M$. Conversely, if an orientation of $M$ is given then we can rescale local 
contact forms $\{\alpha_s\}_{s\in\Sigma}$ annihilating $H$ such that $d\alpha_s^{\wedge n}\wedge\alpha_s$ agree with the orientation. Clearly, such 
rescaled one-forms must coincide on the intersections of domains $U_s$. Thus, they define a global one-form on $M$.

If $n$ is odd then $(d\alpha)^{\wedge n}\wedge\alpha$ is independent of the sign of $\alpha$. Conversely, the sign of $d\alpha^{\wedge n}$ changes if the 
sign of $\alpha$ changes. Thus, similarly to the case of even $n$, we deduce that on the one hand $M$ is canonically oriented, and, on the other hand, $H$ 
is oriented if and only if there is global contact form annihilating $H$.

Suppose that $(M,H,g)$ is oriented. In view of the discussion above we can assume that the orientation of $M$ is given by  $d\alpha^{\wedge n}\wedge\alpha$ and the orientation of $H$ is given by $d\alpha^{\wedge n}$, where $\alpha$ is a global contact form. Then $\alpha$ is given up to a multiplication by a positive function. However, we can choose the unique one which satisfies the normalisation condition \eqref{norm1}. We shall call this form the \emph{canonical contact form} of an oriented contact structure. The canonical contact form satisfies 
\begin{equation}\label{norm2}
(d\alpha)^{\wedge n}(X_1,\ldots,X_{2n})=1.
\end{equation}
where  $(X_1,\ldots,X_{2n})$ is an arbitrary positively oriented orthonormal frame of $H$.

If $(M,H,g)$ is oriented then we shall consider isometries preserving the orientation.

\subsection{Symplectic structure}
Assume that $(M,H,g)$ is an oriented sub-pseudo-Riemannian manifold and let $\alpha$ be the canonical contact form. We introduce
$$
\omega=-d\alpha|_{H}.
$$
Then $\omega$ is a symplectic structure on $H$ canonically defined by $\alpha$.

\begin{prop}\label{prop2}
If $f\colon M\to M$ is an isometry of an oriented sub-pseudo-Riemannian structure then $f^{*}\omega=\omega$.
\end{prop}

The pair $(g,\omega)$ defines the operator $J\colon H\to H$ by the formula 
\begin{equation}\label{defJ}
\omega_q(v,w)=g(J_q(v),w), \qquad q\in M, \quad v,w\in H_q.
\end{equation}
The eigenvalues of $J$ are basic invariants of the structure $(H,g)$ at each point $q\in M$. We shall analyse the structure of eigenspaces of $J$ using the Kronecker theorem that gives normal forms of pencils of matrices. Precisely, we apply the Kronecker theorem to the pair $(g,\omega)$, i.e. to a pair of a symmetric and a skew-symmetric bi-linear forms. For a detailed analysis of this particular case of the Kronecker theorem we refer to \cite{T}. We shall use later the following properties of eigenvalues of $J$:
\begin{enumerate}
\item[(P1)] if $\la$ is an eigenvalue of $J$ then also $-\la$ is;
\item[(P2)] if $\la$ has non-zero real part then $g$ restricted to the corresponding eigenspace $H_\la$ is degenerate: $g|_{H_\la\times H_\la}=0$;
\item[(P3)] if $\la$ is purely imaginary and $g|_{H_\la\times H_\la}\neq 0$ then $\dim H_\la=2$ and $g$ on $H_\la$ is definite; in this case $b=|{\la}|$ 
is called a frequency (c.f. \cite{A}).
\end{enumerate}
Thus, at each $q\in M$ the distribution $H$ decomposes as follows
\begin{equation}\label{decomp}
H=\hat H\oplus \tilde H
\end{equation}
where
\begin{equation}\label{decomp2}
\hat H=H_{\lambda_1}\oplus\cdots\oplus H_{\lambda_s}.
\end{equation}
and all $H_{\lambda_i}$, $i=1,\ldots, s$ are two dimensional and correspond to purely imaginary eigenvalues (some may repeat). Additionally $\tilde H$ is of dimension $2n-2s$ and $\ind g|_{\tilde H\times \tilde H}=n-s$. Moreover $\tilde H$ decomposes further to eigenspaces which are null with respect to $g$ and appear in pairs $H_\la\oplus H_{-\la}$.

According to \eqref{decomp} and \eqref{decomp2} $J$ has the following form
\begin{equation}\label{matrix1}
J=\left(
\begin{matrix}
      0 & -b_1 & \cdots & 0 & 0 & 0 \\
      b_1 & 0 & \cdots & 0 & 0  & 0 \\
      \vdots & \vdots & \ddots & \vdots& \vdots & \vdots\\
      0 & 0 & \cdots & 0 & -b_s &0 \\
      0 & 0 & \cdots & b_s & 0 & 0\\
      0 & 0 & \cdots & 0 & 0 & \tilde J
\end{matrix}\right)
\end{equation}
where $(b_1,\ldots,b_s)$ are frequencies and $\tilde J$ is a matrix of dimension $(2n-2s)\times (2n-2s)$.

Note that if the signature of $g$ is Riemannian then there is no term $\tilde H$ in the decomposition and we have $n$ frequencies that satisfy (c.f. \cite{A})
$$
\prod_{i=1}^nb_i=1,
$$
due to \eqref{norm2}. On the other hand if index of $g$ is odd then $\tilde H$ always appears. In a very particular case it may occur that all eigenvalues 
of $\tilde J$ are real and the corresponding eigenspaces are one dimensional. Then
\begin{equation}\label{matrix2}
\tilde J=\left(
\begin{matrix}
      0 & c_1 & \cdots & 0 & 0 \\
      c_1 & 0 & \cdots & 0 & 0 \\
      \vdots & \vdots & \ddots & \vdots& \vdots\\
      0 & 0 & \cdots & 0 & c_t \\
      0 & 0 & \cdots & c_t & 0
\end{matrix}\right)
\end{equation}
for some $(c_1,\ldots,c_t)$ where $t=n-s$ and
$$
\left(\prod_{i=1}^sb_i\right)\left(\prod_{i=1}^tc_i\right)=1,
$$
due to \eqref{norm2} again. In particular if $(H,g)$ is a sub-Lorentzian structure in dimension 3 then the three properties (P1)-(P3) of $J$ and the normalisation condition \eqref{norm2} imply that $J=\tilde J$ and the two null directions in $H$ are eigenspaces with real eigenvalues $\pm 1$. Therefore, we can choose an orthonormal frame such that $J=\tilde J$ and
$$
\tilde J=\left(
\begin{matrix}
      0 & 1 \\
      1 & 0 
\end{matrix}\right).
$$

\subsection{Reduction}
Let $(H,g)$ be an oriented sub-pseudo-Riemannian contact structure on $M$. Then the symplectic structure $\omega$ reduces $O_{H,g}(M)$ to the sub-bundle $O_{H,g,\omega}(M)$ of frames which put $(g,\omega)$ into the canonical Kronecker form. Thus, $O_{H,g,\omega}(M)$ consists of points $(q; v_1,\ldots, v_{2m}, v_0)\in O_{H,g}$ such that $(v_1,\ldots,v_{2m})$ is a positively oriented, orthonormal basis of $H_q$ and $J_q$ in this basis is given by \eqref{matrix1}. Moreover, we assume that $v_0=X_\alpha(q)$, where $X_\alpha$ is the Reeb vector field corresponding to the canonical contact form. Then the following group acts freely and transitively on $O_{H,g}(M)_q$ 
\begin{eqnarray}\label{group}
\GG_{g,\omega}(q)=\left\{\left(
\begin{matrix}
      A & 0 \\
      0 & 1
\end{matrix}\right) \ |\  A\in O(g_q)\cap Sp(\omega_q) \right\},
\end{eqnarray}
where $O(g_q)$ is the subgroup of $GL(H_q)$ preserving $g_q$ and $Sp(\omega_q)$ is the subgroup of $GL(H_q)$ preserving $\omega_q$. Of course $O(g_q)\simeq O(l,2n-l)$, where $l=\ind(g)$ and $O(l,2n-l)$ is the standard group of matrices preserving a metric of index $l$ and $Sp(\omega_q)\simeq Sp(2n)$, where $Sp(2n)$ is the group of matrices preserving the standard symplectic form given by
\begin{equation}\label{standard_sympl}
\Omega=\left(
\begin{matrix}
      0 & -I_n \\
      I_n & 0 
\end{matrix}\right),
\end{equation}
where $I_n$ is the $n\times n$ identity matrix. Note that automatically $O(g_q)\cap Sp(\omega_q)\subset SO(g_q)$, because the orientation is defined in terms of $\omega_q$.

The intersection $O(g_q)\cap Sp(\omega_q)$ essentially depends on $g$ and $\omega$ at a given point and the groups $\GG_{g,\omega}(q)$ may be not isomorphic for different $q$. Actually, we shall show later that the dimension of $\GG_{g,\omega}(q)$ depends on the decomposition of $J_q$ into the sum of eigenspaces.

\section{Compatibility condition}\label{sec3}

\subsection{Isometries of compatible structures}
We will consider a particular class of oriented contact sub-pseudo-Riemannian structures such that $g$ and $\omega$ are compatible. One expects that the most symmetric structures are among this class.
\begin{defin}
Let $(M,H,g)$ be an oriented sub-pseudo-Riemannian manifold and let $\omega$ be the corresponding symplectic structure on $H$. Then $g$ and $\omega$ are \emph{compatible} if in a neighbourhood of any $q\in M$ there is a frame which is mutually orthonormal with respect to $g$ and symplectic with respect to $\omega$. The sub-pseudo-Riemannian structure satisfies the \emph{compatibility condition} if $g$ and $\omega$ are compatible.
\end{defin}
Note that in the case of compatible structures with $g$ being Riemannian, $J$ is an almost complex structure on $H$. Similarly, in the case of compatible 
structures with $\ind(g)=\frac{1}{2}\rk H$, $J$ is a para-CR structure, provided that there are no purely imaginary eigenvalues of $J$. In general, 
the compatibility condition can be expressed in terms of frequencies.
\begin{prop}\label{prop3}
An oriented sub-pseudo-Riemannian structure satisfies the compatibility condition if and only if the frequencies in \eqref{matrix1} satisfy $b_i=1$, $i=1,\ldots,s$ and $\tilde J$ is of the form \eqref{matrix2} with $c_i=1$, $i=1,\ldots,t$.
\end{prop}
\begin{proof}
Follows directly from the definition.
\end{proof}

The bundle $O_{H,g,\omega}(M)$ for a structure $(H,g)$ satisfying the compatibility condition is the bundle of frames that are mutually orthonormal with respect to $g$ and symplectic with respect to $\omega$. Proposition \ref{prop3} implies that under the compatibility condition all $\GG_{g,\omega}(q)$, $q\in M$, are isomorphic, because $J_q$ depends smoothly on $q$ and $M$ is connected. Thus $O_{H,g,\omega}(M)$ is a principal bundle with the structure group isomorphic to $\GG_{g,\omega}(q)$ for any fixed $q\in M$. The structure group will be simply denoted $\GG_{g,\omega}$. Moreover, Proposition \ref{prop2} implies that the embedding \eqref{embed1} restricted to the component of identity $\II_0(M,H,g)$ takes values in $O_{H,g,\omega}(M)$. Precisely, fixing $(q; v_1,\ldots,v_{2n},v_0)\in O_{H,g,\omega}$ we get that
\begin{equation}\label{embed2}
f\longmapsto(f(q); f_*(v_1),\ldots,f_*(v_{2n}),f_*(v_0))
\end{equation}
defines an embedding of $\II_0(M,H,g)$ to $O_{H,g,\omega}(M)$. This embedding permits to prove
\begin{thm}\label{thm3}
Let $(M,H,g)$ be an oriented contact sub-pseudo-Rie\-man\-nian manifold satisfying the compatibility condition. Then 
$$
\dim\II(M,H,g)\leq 2n+1+s^2+(n-s)^2,
$$
where $\dim M=2n+1$ and $s=\frac{1}{2}\rk\hat H$ is the multiplicity of $i=\sqrt{-1}$ as an eigenvalue of the endomorphism $J$. Moreover, the parity of $n-s$ equals to the parity of $\ind(g)$.  
\end{thm}
\begin{proof}
We recall that $g$ restricted to any two-dimensional component $H_\la$ of $\hat H$ in the decomposition \eqref{decomp2} is definite. Additionally $g$ restricted to $\tilde H$ has index equal to $\frac{1}{2}\rk\tilde H$. Thus 
$$
\ind(g)=\frac{1}{2}\rk\tilde H\mod 2
$$
and since $\frac{1}{2}\rk\tilde H=n-s$ the last statement of the Theorem follows.

Therefore, it is sufficient to compute the dimension of $\GG_{g,\omega}$ in order to complete the proof, because the existence of the embedding \eqref{embed2} implies
$$
\dim\II_0(M,H,g)\leq\dim M+\dim\GG_{g,\omega},
$$
and $\dim\II(M,H,g)=\dim\II_0(M,H,g)$. The result follows from the following general Lemma that will be also used later in the proof of Theorem \ref{thm1}.
\begin{lem}\label{lem1}
Let $s=\frac{1}{2}\rk \hat H$ and $t=\frac{1}{2}\rk \tilde H$, where $\hat H$ and $\tilde H$ are defined by the decomposition \eqref{decomp} of the operator $J$ for a pair $(g,\omega)$ of arbitrary non-degenerate symmetric and skew-symmetric bi-linear forms on $H$, $\rk H=2n$. Then
$$
\dim\left(O(g)\cap Sp(\omega)\right)=s^2+t^2.
$$
\end{lem}
\begin{proof}
We shall consider the Lie algebra $\g$ of $O(g)\cap Sp(\omega)$, because $\dim\g=\dim\left(O(g)\cap Sp(\omega)\right)$. Let $A=(a_{i,j})_{i,j=1,\ldots,2n}\in\mathfrak{g}$. Then, according to the decomposition \eqref{decomp} $A$ decomposes into the following block form
$$
A=\left(
\begin{matrix}
      B & D\\
      D' & C
\end{matrix}\right)
$$
where $A$ is of dimension $2s\times 2s$, $D$ is of dimension $2s\times 2t$, $D'$ is of dimension $2t\times 2s$ and $C$ is of dimension $2t\times 2t$. Now, since $A$ preserves the eigenspces of $J$ and $g$ restricted to the eigenspaces, we get from properties (P2) and (P3) of $J$ that $D=D'=0$. Thus we shall estimate the possible number of independent entries of $B$ and $C$.

Let us consider $B$ first. In order to get an estimate we can assume that all $b_i=1$. Otherwise $B$ would decompose into smaller blocks. So, we can choose a basis in $\hat H$ such that $g$ is diagonal and $\omega$ is a standard symplectic form. Then, on the one hand $B$ is completely determined by entries above the diagonal, because $B\in\mathfrak{so}(\hat l,2s-\hat l)$, where $\hat l=\ind g|_{\hat H\times\hat H}$. On the other hand $B$ is completely determined by the entries above the anti-diagonal (including the anti-diagonal itself), because $B\in\mathfrak{sp}(2s)$. Thus, $B$ has $s^2$ independent entries.

Now, let us consider $C$. We have $\ind g|_{\tilde H\times\tilde H}=\frac{1}{2}\rk\tilde H=t$. Thus, we can assume that $g$ is diagonal 
$$
g|_{\tilde H\times\tilde H}=\diag(\underbrace{-1,\ldots,-1}_t,\underbrace{+1,\ldots,+1}_t).
$$
Moreover, due to (P2), we can assume $\omega|_{\tilde H\times\tilde H}$ is given by a non-degenerate skew-symmetric matrix of the form
$$
\omega|_{\tilde H\times\tilde H}=\left(
\begin{matrix}
      0 & \tilde\omega_{12}\\
      -\tilde\omega^T_{12} & 0
\end{matrix}\right),
$$
where $\tilde\omega_{12}$ is of dimension $t\times t$ (c.f. the normal forms in \cite{T}). Let
$$
C=\left(
\begin{matrix}
      C_{11} & C_{12}\\
      C_{21} & C_{22}
\end{matrix}\right),
$$
where all $C_{ij}$ are of dimension $t\times t$. Then $C_{11}$ and $C_{22}$ are skew-symmetric and $C_{12}=C_{21}^T$ due to $C\in\mathfrak{so}(t,t)$. Moreover $C_{11}=\tilde\omega_{12}C_{22}\tilde\omega_{12}^{-1}$ and $C_{12}=\tilde\omega_{12}C_{12}^T\tilde\omega_{12}^{-1}$ due to $C\in \mathfrak{sp}(\omega)$. Thus, $C$ has at most $t^2$ independent entries.

Finally let us notice that the maximal dimensions are attained if all $b_i=1$ and $\tilde J$ is of the form \eqref{matrix2} with all $c_i=1$. 
\end{proof}
\end{proof}

\subsection{Left invariant structures on the Heisenberg group}
We will show that the upper bound on the dimension of the group of isometries from Theorem \ref{thm3} is attained. In particular, taking into account the parity of $\ind(g)$, we will show that there are structures with the isometry groups of dimensions as in Theorem \ref{thm1} and formula \eqref{est3}.

To this end we consider left-invariant structures on the Heisenberg group. We recall that the Heisenberg group is realised as the space $\R^{2n+1}$ with the contact distribution $H$ defined as follows. Suppose we have coordinates $x_1,\ldots,x_n,y_1,\ldots,y_n,z$ on $\R^{2n+1}$ which will be denote by $(x,y,z)$ for short. Let
\begin{equation}\label{vfields}
 X_i=\frac{\partial}{\partial x_i}+\frac{1}{2}y_i\frac{\partial}{\partial z}, \qquad
 Y_i=\frac{\partial}{\partial y_i}-\frac{1}{2}x_i\frac{\partial}{\partial z},
\end{equation}
$i=1,\ldots,n$. Define $H$ to be 
$$
H=\spn\{X_1,Y_1,\ldots,X_n,Y_n\} .
$$
We equip $(\R^{2n+1}, H)$ with metric $g$ by declaring the frame $(X_1,Y_1,\allowbreak\ldots,X_n,Y_n)$ to be orthonormal and such that
$$
g(X_i,X_i)=t_i,\qquad g(Y_i,Y_i)=s_i,
$$
where $t_i,s_i \in \{-1,1\}$ depending on the signature of $g$. The vector fields \eqref{vfields} are left invariant fields with respect to the standard multiplication on the Heisenberg group
\begin{equation}\label{mult}
\begin{aligned}
&(x_1,\dots,x_n,y_1,\dots,y_n,z)*(x'_1,\dots,x'_n,y'_1,\dots,y'_n,z')= \\
&(x_1+x'_1,\dots.x_n+x'_n,y_1+y'_1,\dots,y_n+y'_n,z+z'+\frac{1}{2}\sum_{i=11}^{n}(y_ix'_i-y'_ix_i)).
\end{aligned}
\end{equation}
The symplectic structure on $H$ is the standard one
$$
\omega=\sum_{i=1}^n dx_i\wedge dy_i.
$$

Take a matrix $\sigma\in Sp(\omega)\cap O(g)$. We will show that the map $f_\sigma\colon\R^{2n+1}\to\R^{2n+1}$ defined by
\begin{equation}\label{isom1}
f_\sigma(x,y,z)=(\sigma\cdot(x,y)^T,z)
\end{equation}
is an isometry. Denote $f_\sigma=(f_\sigma^1,\ldots,f_\sigma^{2n},f_\sigma^{2n+1})$. Then
\begin{equation}\label{isom2}
f_\sigma^i(x,y,z)=\sum_{j=1}^n(\sigma_{i,j}x_j+\sigma_{i,n+j} y_j)
\end{equation}
for $i=1,\ldots,2n$. First we have
\begin{lem}\label{lem2}
For any $\sigma\in Sp(\omega)$ 
$$
{f_\sigma}_*(X_i)(x,y,z)= \sum_{j=1}^n \sigma_{j,i} X_j(f_\sigma(x,y,z))+ \sum_{j=1}^n \sigma_{n+j,i}Y_j(f_\sigma(x,y,z))
$$
and 
$$
{f_\sigma}_*(Y_i)(x,y,z)= \sum_{j=1}^n \sigma_{j,n+i} X_j(f_\sigma(x,y,z)) + \sum_{j=1}^n \sigma_{n+j,n+i}Y_j(f_\sigma(x,y,z)).
$$
In particular, $f_\sigma$ preserves $H$.
\end{lem}
\begin{proof}
We will prove the first equality only. Using \eqref{isom2} we directly compute
$$
{f_\sigma}_*(X_i)=\sum_{j=1}^n\sigma_{j,i}\frac{\partial}{\partial x_j}+ \sum_{j=1}^n \sigma_{n+j,i}\frac{\partial}{\partial y_j}+\frac{1}{2}y_i\frac{\partial}{\partial z}.
$$
Now, it is enough to show that 
$$
\sum_{j=1}^n \sigma_{j,i}f_\sigma^{n+j}(x,y,z)-\sum_{j=1}^n \sigma_{n+j,i}f_\sigma^j(x,y,z)=y_i.
$$
However, using \eqref{isom2} again, we have
$$
\begin{aligned}
&\sum_{j=1}^n\sigma_{j,i}f_\sigma^{n+j}(x,y,z)-\sum_{j=1}^n\sigma_{n+j,i}f_\sigma^j(x,y,z)= \\ 
&\sum_{j,k=1}^n(\sigma_{n+j,k}\sigma_{j,i}-\sigma_{j,k}\sigma_{n+j,i})x_k + 
 \sum_{j,k=1}^n(\sigma_{n+j,n+k} \sigma_{j,i}-\sigma_{j,n+k} \sigma_{n+j,i})y_k\\
\end{aligned}
$$
and the lemma follows from the fact that $\omega$ is the standard symplectic form, i.e.~$\sigma\,\Omega\,\sigma^T= \Omega$, where $\Omega$ is given by \eqref{standard_sympl}.
\end{proof}

Now, we can prove the following
\begin{prop}\label{prop4}
The group of orientation preserving isometries of the left-invariant contact sub-pseudo-Riemannian structure defined above on the Heisenberg group is isomorphic to 
\begin{equation}\label{flatgroup}
\R^{2n+1} \ltimes \left(Sp(\omega)\cap O(g)\right).
\end{equation}
\end{prop}
\begin{proof}
If $\sigma\in Sp(\omega)\cap O(g)$ then the formulae for ${f_\sigma}_*(X_i)$ and ${f_\sigma}_*(Y_i)$ in Lemma \ref{lem2} imply that $f_\sigma$ is an isometry. Thus any $\sigma\in Sp(\omega)\cap O(g)$ defines an isometry of $(H,g)$ and we get the second factor in \eqref{flatgroup}. The first factor in \eqref{flatgroup} comes from left translations. There can not be more isometries due to the embedding \eqref{embed2}.
\end{proof}

\begin{rem}
Let us remark that the full group of isometries is isomorphic to the product $\R^{2n+1} \ltimes \left(\tilde{Sp}(\omega)\cap O(g)\right)$ where $\tilde{Sp}(\omega)$ is 
the group preserving $\omega$ up to the sign.
\end{rem}

\section{Regularity condition}\label{sec4}

\subsection{Isometries of regular structures}
Before proceeding to the general case announced in Theorem \ref{thm1} we will describe a class of sub-pseudo-Riemannian structures which generalize those satisfying the compatibility condition but, at the same time, simple enough so that the isometry groups can be explicitly computed.

\begin{defin}
Let $(M,H,g)$ be a contact sub-pseudo-Riemannian manifold of dimension $2n+1$. The metric $(H,g)$ is said to satisfy \emph{the regularity condition} if there exists a global orthonormal frame $X_1,\dots,X_{2n}$ with respect to which the symplectic form  $\omega$ on $H$ can be written as
$$
\omega=\sum_{i=1}^{n}b_i\alpha^i\wedge \alpha^{n+i}
$$
where $\alpha^1,\ldots,\alpha^{2n}$ is the co-frame dual to $X_1,\dots,X_{2n}$, and $b_1,\dots,b_n$ are smooth functions such that there exist positive integers $k_1,\dots,k_r$, $k_1+\dots+k_r=n$, for which
\begin{equation}\label{reg_b}
b_1=\cdots =b_{k_1} \neq b_{k_1+1}=\cdots =b_{k_1+k_2}\neq \cdots \neq
b_{k_1+\dots+k_{r-1}+1}=\cdots =b_n
\end{equation}
holds on the whole of $M$.
\end{defin}

Note that any sub-Riemannian or sub-Lorentzian structure fulfils the regularity condition at least on an open subset of $M$.
Clearly, the functions $b_i$ are related to either real or purely imaginary eigenvalues of the operator $J$. In fact, if $(M,H,g)$ is regular then $\tilde J$ has necessarily form \eqref{matrix2}. Let
$$
H^i=\spn\{X_j, X_{n+j}\ |\ k_1+\ldots+k_{i-1}+1\leq j\leq k_1+\ldots+k_i\}.
$$
Then all $H^i$, $i=1,\ldots,r$, are invariant with respect to $J$ and $H$ splits into the Whitney sum 
$$
H=H^1\oplus\cdots\oplus H^r.
$$
Moreover, the groups $\GG_{g,\omega}(q)$, $q\in M$, split into the direct product 
\begin{equation}\label{group_red}
\GG_{g,\omega}(q)\simeq\left(Sp(\omega|_{H^1_q})\cap O(g|_{H^1_q})\right) \oplus \cdots \oplus \left(Sp(\omega|_{H^r_q})\cap O(g|_{H^r_q})\right).
\end{equation}
All groups $\GG_{g,\omega}(q)$ are isomorphic under the regularity condition and will be shortly denoted $\GG_{g,\omega}$. Consequently, the bundle $O_{H,g,\omega}(M)$ admits a reduction to a $\GG_{g,\omega}$-structure which can be realized as the set of all such frames $(q;v_1,\allowbreak\ldots, v_{2n},v_0)\in O_{H,g,\omega}$ that
$$
v_j,v_{n+j}\in H^i_q,\qquad k_1+\ldots+k_{i-1}+1\leq j\leq k_1+\ldots+k_i,
$$
for $i=1,\ldots,r$.
The presented considerations lead to the following
\begin{thm}\label{thm4} 
Let $(M,H,g)$ be an oriented contact sub-pseudo-Rie\-man\-nian manifold satisfying the regularity condition. Then 
$$
\dim \II(M,H,g)\leq 2n+1+s_1^2+(k_1-s_1)^2+\dots+s_r^2+(k_r-s_r)^2,
$$
where $s_i=\frac{1}{2}\rk(H^i\cap\hat H)$.
\end{thm}
\begin{proof}
Indeed, $\dim\II(M,H,g)\leq\dim M+\dim\GG_{g,\omega}$ and the result follows from Lemma \ref{lem1} applied to each factor of $\GG_{g,\omega}$ separately.
\end{proof}

\subsection{Left invariant regular structures}
Now we are going to show that the upper bound on the dimension of the isometry group given in Theorem \ref{thm4} is attained. To this end, fix positive real numbers $b_i$, $i=1,\dots,n$, as in \eqref{reg_b} and define the following multiplication on $\R^{2n+1}$
\begin{equation}
\begin{aligned}\label{mult2}
&(x_1,\dots,x_n,y_1,\dots,y_n,z)*(x'_1,\dots,x'_n,y'_1,\dots,y'_n,z')= \\  
&(x_1+x'_1,\dots,x_n+x'_n,y_1+y'_1,\dots,y_n+y'_n,\frac{1}{2}\sum_{i=1}^n b_i(y_ix'_i-y'_ix_i)).
\end{aligned}
\end{equation}
The multiplication \eqref{mult2} can be treated as a deformation of the standard multiplication \eqref{mult}. Now it is not difficult to see that the left invariant vector fields with respect to this multiplication are given by formulae 
\begin{equation}\label{vfields2}
X_i=\frac{\partial}{\partial x_i}+\frac{b_i}{2}y_i\frac{\partial}{\partial z},\qquad Y_i=\frac{\partial}{\partial y_i}-\frac{b_i}{2}x_i\frac{\partial}{\partial z}.
\end{equation}
Let $H=\spn\{X_1,Y_1,\ldots,X_n,Y_n\}$ and define metric $g$ by declaring the basis $X_1,Y_1,\dots,X_n,Y_n$ to be orthonormal with 
$$
g(X_i,X_i)=p_i,\qquad g(Y_i,Y_i)=r_i,
$$
where $p_i,r_i\in\{-1,+1\}$ depending on the index of the metric, $i=1,\dots,n$. It clear that the canonical contact form is 
$$
\alpha=dz-\sum_{i=1}^{n}(1/2)b_i(y_i dx_i-x_i dy_i)
$$
and 
$$
\omega=\sum_{i=1}^{n}b_idx_i\wedge dy_i.
$$
It follows from the construction that the left translations with respect to \eqref{mult2} are isometries of $(\R^{2n+1},H,g)$, because vector fields \eqref{vfields2} are left invariant. Moreover, any $\sigma\in\GG_{g,\omega}$ decomposes according to the splitting \eqref{group_red}. Performing similar calculations as in Lemma \ref{lem2} for each factor of this decomposition one can prove
\begin{prop}
The group of orientation preserving isometries of the left invariant contact sub-pseudo-Riemannian structure $(H,g)$ constructed above on $\R^{2n+1}$ is isomorphic to
$$
\R^{2n+1}\ltimes \GG_{g,\omega},
$$
where $\GG_{g,\omega}$ is given by \eqref{group_red}.
\end{prop}

\section{General case}\label{sec5}
\subsection{Symbol algebra}
Let $(M, H, g)$ be an oriented contact sub-pseudo-Riemannian manifold. Let $\g(H)(q)$ be the symbol algebra of $H$ at point $q\in M$. It is a two-step nilpotent graded Lie algebra
$$
\g(H)(q)=\g_{-1}(q)\oplus\g_{-2}(q)
$$
where
$$
\g_{-1}(q)=H_q, \qquad \g_{-2}(q)=T_qM/H_q.
$$
The Lie bracket $\g_{-1}(q)\wedge\g_{-1}(q)\to\g_{-2}(q)$ is defined in terms of the Lie bracket of vector fields on $M$ as follows. Let $v,w\in \g_{-1}(q)$ and let $X_v$ and $X_w$ be two extensions of $v$ and $w$, respectively, to sections of $H$ in a neighbourhood of $q$. Then
$$
[v,w]=[X_v,X_w](q) \mod H_q
$$
does not depend on the extension and defines the Lie bracket in $\g(H)(q)$. Clearly, the Lie algebra $\g(H)(q)$ does not depend on $q$. Moreover the dual space $\g_{-2}(q)^*$ can be identified with $H^\perp_q\subset T^*_qM$ spanned by the contact form $\alpha_q$. It follows that
$$
\alpha_q([v,w])=\omega_q(v,w),
$$ 
i.e. the Lie algebra structure is determined by the symplectic form $\omega$.

The symbol algebra $\g(H,g)(q)$ of $H$ equipped with $g$ at point $q\in M$ is defined as follows
$$
\g(H,g)(q)=\g(H)(q)\oplus\g_0(q)
$$
where $\g_0(q)$ is the algebra of matrices $A\in\mathfrak{gl}(\g(H)(q))$ preserving the metric $g$, i.e. 
$$
g(Av,w)+g(v,Aw)=0
$$
and the Lie bracket on $\g_{-1}(q)$, i.e.
$$
[Av,w]+[v,Aw]=A[v,w].
$$
Since the Lie bracket is encoded in terms of $\omega$ it follows that $\g_0(q)$ is the Lie algebra of the Lie group $\GG_{g,\omega}(q)$ and actually can be thought of as a sub-algebra of $\mathfrak{gl}(\g_{-1}(q))$ . Defining
$$
[A,v]=Av
$$
for $v\in\g_{-1}(q)$ we get that $\g(H,g)(q)$ is a graded Lie algebra. We refer to \cite{Ta} for more information on the symbol algebras of distributions and related structures.

\subsection{Prolongation}
The first prolongation of $\g(H,g)(q)$ is defined as
$$
\pr_1(\g(H,g)(q))=\g(H,g)(q)\oplus \g_1(q),
$$
where $\g_1(q)$ is the set of all Lie algebra derivations $\g(H)\to\g(H,g)$ increasing the gradation by $1$, i.e. any $A\in \g_1(q)$ maps $\g_{-1}(q)$ to $\g_0(q)$ and $\g_{-2}(q)$ to $\g_{-1}(q)$ such that
\begin{equation}\label{jacobi}
A([v,w])=A(v)w-A(w)v
\end{equation}
for all $v,w\in \g_{-1}(q)$. Note that $\dim\g_{-2}=1$ thus for any $A\in\g_1(q)$ the image $A(\g_{-2})$ is a one- or zero-dimensional subspace of $\g_{-2}$.

Higher prolongations of $\g(H,g)(g)$ are defined by induction, similarly to the first prolongation, as Lie algebra derivations increasing the gradation by $k\in\N$. We get
$$
\pr\g(H,g)(q)=\g(H,g)(q)\oplus\bigoplus_{k\in\N}\g_k(q)
$$
and one equips $\pr\g(H,g)(q)$ with the structure of a graded Lie algebra in a natural way. However we shall not describe the structure in detail because we have the following
\begin{lem}\label{lem3}
The first prolongation of $\g(H,g)(q)$ is trivial. Consequently
$$
\pr\g(H,g)(q)=\g(H,g)(q)
$$
\end{lem}
\begin{proof}
Let $\alpha^*_q\in\g_{-2}(q)$ be a vector dual to the contact form $\alpha_q$, i.e. $\alpha_q(\alpha_q^*)=1$. Choose $A\in \g_1(q)$ and denote $v_A=A(\alpha^*_q)$. Let $(v_1,\ldots,v_{2n})$ be an orthonormal basis of $H_q$ that puts $J$ into the canonical Kronecker form \eqref{matrix1}. Then \eqref{jacobi} reads
\begin{equation}\label{jacobi2}
A(v_i)v_j-A(v_j)v_i=g(J(v_i),v_j)v_A.
\end{equation}
Since $(v_1,\ldots,v_{2n})$ is orthonormal it follows that all $A(v_i)$, $i=1,\ldots,2n$, are orthonormal matrices in $\mathfrak{so}(l,2n-l)$. Now, for a fixed value of $v_A$ there is unique $A$ that solves \eqref{jacobi2} in $\mathfrak{so}(l,2n-l)$, where $l=\ind(g)$. This follows from the uniqueness of the Levi-Civita connection of a pseudo-Riemannian metric which is equivalent to the algebraic fact that the system
\begin{equation}\label{jacobi3}
A(v_i)v_j-A(v_j)v_i=0
\end{equation}
has unique solution $A=0$ in the algebra $\mathfrak{so}(l,2n-l)$. The unique solution to \eqref{jacobi2} is of the form
$$
A=\frac{1}{2}\sum_{i=1}^{2k}(v_A\cdot J(v_i)^T) v_i^*
$$
where $(v_1^*,\ldots,v_{2n}^*)$ are dual to $(v_1,\ldots,v_{2n})$ with respect to $g$ and $v_A\cdot J(v_i)^T=A(v_i)$ is a rank-one square matrix $A(v_i)=(a^i_{jk})_{j,k=1,\ldots,2n}$ with entries $a^i_{jk}=v_j^*(v_A)v_k^*(J(v_i))$. Now, since all $A(v_i)$ are orthonormal it follows that $v_j^*(v_A)v_j^*(J(v_i))=0$ for any $j=1,\ldots,2n$. But, for any $i$ there is $j$ such that $v_j^*(J(v_i))\neq 0$. Thus we get that $v_j^*(v_A)=0$, for $j=1,\ldots,2n$. Consequently, $v_A=0$. This reduces \eqref{jacobi2} to \eqref{jacobi3}. Hence $A=0$, because this is the unique solution to \eqref{jacobi3} as was explained above.
\end{proof}

Now we are able to apply Theorem 1 of \cite{Kr2} and get
\begin{thm}\label{thm6}
Let $(M,H,g)$ be an oriented contact sub-pseudo-Rie\-mannian manifold. Then the dimension of the algebra of the infinitesimal symmetries of $(M,H,g)$ is estimated from above by
$$
\dim M+\inf_{q\in M}\dim\GG_{g,\omega}(q).
$$
\end{thm}
\begin{proof}
We have $\pr\g(H,g)(q)=\g(H,g)(q)$. Thus $\dim\pr\g(H,g)(q)=\dim M+\dim\GG_{g,\omega}(q)$ since $\g_0(q)$ is the Lie algebra of $\GG_{g,\omega}(q)$. All the prolongations are finite. Therefore, by \cite{Kr2}, we have that the dimenison of the algebra of infinitesimal symmetries is estimated from above by $\inf_{q\in M}\dim\pr\g(H,g)(q)$.
\end{proof}

\subsection{Proof of Theorem \ref{thm1}} If $(M,H,g)$ is an oriented sub-pseudo-Riemannian manifold then it suffices to consider isometries preserving the orientation because $\dim\II(M,H,g)=\dim\II_0(M,H,g)$. The dimension of $\II_0(M,H,g)$ equals to the dimension of the algebra of infinitesimal isometries. Therefore we can apply Theorem \ref{thm6}. The maximal possible dimension of $\GG_{g,\omega}$ is computed in Lemma \ref{lem1}.

If $(M,H,g)$ is not oriented then we consider a double cover $\tilde M$ of $M$ consisting of pairs $(q,\alpha_q)$ where $q\in M$ and $\alpha_q$ is one of the two normalised co-vectors in $T^*_qM$ annihilating $H(q)$. Then $\tilde M$ carries a canonical lift $(\tilde H, \tilde g)$ of the structure $(H,g)$ and the structure $(\tilde M,\tilde H,\tilde g)$ is oriented, because $(q,\alpha_q)\mapsto\alpha_q$ defines a global contact form on $\tilde M$ annihilating $\tilde H$. Moreover, any isometry of the original structure $(M,H,g)$ defines an isometry of $(\tilde M,\tilde H,\tilde g)$ and thus $\dim\II(M,H,g)\leq\dim\II(\tilde M,\tilde H,\tilde g)$. Therefore, the estimate in the not oriented case follows from the estimate in the oriented case. 

\vskip 2ex
\paragraph{\bf Acknowledgements.} The work of Wojciech Kry\'nski has been partially supported by the Polish National Science Centre grant DEC-2011/03/D/ST1/03902.

\end{document}